\documentclass[12pt,epsfig]{amsart}
\usepackage{latexsym}
\usepackage{amsmath,amscd,amssymb}
\usepackage{graphics}

 2

\include{plaquettemacro}
\usepackage{hyperref}
\newtheorem{theorem}{Theorem}[section]
\newtheorem{proposition}{Proposition}[section]

\newtheorem{definition}[theorem]{Definition}

\newtheorem{remark}[theorem]{Remark}
\numberwithin{equation}{section}

\newcommand{\mbr}{{\mathbb R}}

\newcommand{\mM}{{\mathcal{M}}}

\title{Symplectic cohomology (defined by Krigel and Michor) of the loop space $L\mbr^n$ }
\author{Pradip Kumar}
\address{Department of Mathematics,
Harish Chandra Research Institute\\ Allahabad 211019, Uttarpradesh\\ India} \email{pmishra@hri.res.in}
\begin{document}

\maketitle
\begin{abstract}
For a finite dimensional symplectic manifold $(M,\omega)$ with a symplectic form $\omega$,  corresponding loop space ($LM=C^\infty(S^1,M)$) admits a weak symplectic form $\Omega^\omega$.  We  prove that the loop space over $\mbr^n$ admits Darboux chart for the weak symplectic structure $\Omega^\omega$.  Further, we show that inclusion map from the symplectic cohomology (as defined by Kriegl and Michor \cite{KM}) of the loop space over $\mathbb R^n$  to the De Rham cohomology of the loop space is an isomorphism.
\end{abstract}
\section{Introduction}
Let $(M,\omega)$ be a finite dimensional symplectic manifold. Loop space $LM:=C^\infty(S^1,M)$ is a nuclear Fr\'{e}chet manifold.  For manifold structure on $LM$ we refer \cite{Michor},\cite{KM}.  For $\gamma\in LM$, $T_\gamma LM= \Gamma_{S^1}(\gamma^*TM)$.  As $M$ is oriented, we can identify $\Gamma_{S^1}(\gamma^*TM)$ with $L\mbr^n$.

For $X, Y\in T_\gamma LM=\Gamma_{S^1}(\gamma^*TM)$, $X(t), Y(t)\in \gamma^*(TM)$,  define:
\begin{equation}\label{defn:symlecticfromsymlectic}\Omega^\omega_\gamma(X,Y)= \int_0^1\omega_{\gamma(t)}(X(t), Y(t)) dt
\end{equation}

Recall, by a weak symplectic form $\Omega$ on a Fr\'{e}chet manifold $\mM$, we mean that the induced map $\Omega^\vee: T_p \mM\to T^*_p\mM$ is an injective map and $\Omega$ is a closed 2 form.  By a strong symplectic form $\Omega$ on an infinite dimension manifold $\mM$ (Banach or Fr\'{e}chet manifold), we mean that the corresponding map $\Omega^\vee:T\mM\to T^*\mM$ is a topologically isomorphism.

In section 3, we will see that $\Omega^\omega$ as defined above in equation \ref{defn:symlecticfromsymlectic} is a closed form on $L\mbr^n$. We will see that it is not strong symplectic form  rather it is only a weak symplectic form. We mention here that there does not exist any strong symplectic form on $LM$. For example, in the case of $L\mbr^n$,  where $T_\gamma L\mbr^n\approx L\mbr^n$ but $T_\gamma^*L\mbr^n$, the dual space of $L\mbr^n$,  is the set of all $\mbr^n$ valued distribution on circle [Page 39,\cite{stacey2}].  Therefore  $T_\gamma L\mbr^n$ can not be topologically isomorphic to $T^*_\gamma L\mbr^n$.  In general if $\Omega$ is a symplectic form on Fr\'{e}chet manifold, it  can not be strong symplectic. However it may be weak symplectic ($\Omega^\vee$ injective) or quasi symplectic (kernel of $\Omega^\vee$ is finite dimension). For example, we can see that $\Omega^g$ (symplectic structure on loop space induced from a Riemannian metric $g$, defined in \cite{wurzbacher}) is quasi symplectic for $LM$ and weak symplectic for $L_pM$ for any $p\in M$ ( here $L_pM:=\{\gamma\in LM: \gamma(0)\in M\}$ is a submanifold of $LM$).

If $(M,\sigma)$ is a finite dimensional symplectic manifold, Darboux theorem states that every point in $M$ has a coordinate neighborhood $N$ with coordinate functions $(x_1,.., x_n, y_1,.., y_n)$ such that $\sigma= \sum_{i=1}^n dx_i\wedge dy_i$ on $N$.  This chart is called Darboux chart around $p\in M$. In the infinite dimension setting we have the following definitions:

Let $\mM$ be a Fr\'{e}chet manifold modeled on a Fr\'{e}chet space $E$. By a Darboux chart around $p\in \mM$, we mean a coordinate chart $(\mathfrak{U},\Phi)$ around $p$ such that there exists a bounded alternating bilinear map $\mathcal{F}$ on $E$ such that for $v_1,v_2\in T_q\mM$, $\sigma_{\Phi(q)}(v_1, v_2)= F(d\Phi_q(v_1),d\Phi_q(v_2))$ for every $q\in \mathfrak{U}$. This chart around $p$ is called Darboux chart.  We say that $(\mM,\sigma)$ admits Darboux chart if for every $p\in \mM$ there is a Darboux chart around $p\in \mM$.

Our aim in this article  is to address the following questions:
\begin{enumerate}
\item Does the weak symplectic manifold $(L\mbr^n,\Omega^\omega)$ admit Darboux chart.  In theorem \ref{thm:isotopyalpha}, we prove that if there is an isotopy $\phi$ such that $\phi^*(\omega_s)= \omega_0$, then there exists an isotopy on loop space $\phi^L$ such that $(\phi^L)^*(\Omega^{\omega_s})=\Omega^{\omega_0}$. This prove that $(L\mbr^n,\Omega^\omega)$ admits Darboux chart.
\item As a corollary of existence of Darboux chart on loop space, in section 5, we prove that the inclusion map from the  symplectic cohomology (as defined by Kriegl and Michor \cite{KM}) of the loop space over $\mathbb R^n$  to the De Rham cohomology of $L\mbr^n$ is an isomorphism.
\end{enumerate}
At the last we will make remark on above questions for the general loop space $LM$ in place of $L\mbr^n$.

We will start this article by a discussion about the calculus which we will be using for the loop space. In section 2,  we will describe a calculus on loop space.  We will describe an isotopy of the loop space.  Section 3,4 are about the symplectic form $\Omega^\omega$ on $L\mbr^n$ and proof for existence of Darboux chart on $L\mbr^n$.   In section 5, we will revise definition of the symplectic cohomology defined by Kriegl and Michor in \cite{KM}.  This definition of symplectic cohomology is not same as the definition of the symplectic cohomology by Floer.   We show that the  symplectic cohomology of the loop space $L\mbr^n$ defined by Kriegl and Michor \cite{KM} is isomorphic to the De-Rham cohomology.
\\

\textbf{Acknowledgement:}  I wish thank to Dr. Saikat Chatterjee for explaining the weak symplectic structure $\Omega^\omega$ as defined in equation \ref{defn:symlecticfromsymlectic} and  sharing his work (with Prof. Indranil Biswas and Prof. Rukmini Dey) on pre-quantization of a space.


\section{Smooth structure on Loop space and Isotopy}
For a smooth loop space we have to leave the realm of Banach space, but as soon as we are leaving the realm of Banach space, there are various definitions of smooth map on locally convex spaces (see at the end of first chapter of \cite{KM}). Even for our case $L\mbr^n$ (a Fr'{e}chet space), there are three inequivalent definitions of differentiability (3.1 \cite{stacey2}). Kriegl and Michor define a smooth linear map between locally convex space as a bounded linear map \cite{KM}, but then smooth linear map may not be continuous, so for removing this intricacy they took finest locally convex topology such that bounded map becomes continuous. This finest locally convex topology of given space is called the Bornology.  Finally they define a smooth map as a map which takes smooth curve to smooth curve. In the Fr\'{e}chet space we do not have much problem because the Fr\'{e}chet space topology is the bornology of the space. This notion of smoothness agrees with the  notion of  R.S. Hamilton approach in \cite{hamilton}.  For our object of study,  the loop space,  Andrew Stacey has provided a detail discussion in \cite{stacey2} based on Kriegl and Michor \cite{KM} calculus. Following the proposition 3.1 of \cite{stacey2}, we have the definition: $c: \mbr\to L\mbr^n$ is smooth if and only if induced map $c^\vee:\mbr\times S^1\to \mbr^n$ is smooth.  A map $f: (U\subset L\mbr^n)\to L\mbr^n$ is smooth if and only if $f$ maps smooth curves to smooth curve. Also if we want to find the derivative of $f$, we have a consistent definition of derivative as directional derivative following by Hamilton \cite{hamilton}.

Now let us define the isotopy which we will use in later sections.
\begin{definition}
Smooth map $\phi: \mbr\times L\mbr^n \to L\mbr^n$ is called an isotopy if each $\phi_s: L\mbr^n\to L\mbr^n$ is a diffeomorphism and $\phi_0= Id_{L\mbr^n}$.
\end{definition}

Let $\phi:\mbr\times \mbr^n\to \mbr^n$ is an isotopy then define
\begin{equation}\phi^L:\mbr\times L\mbr^n\to L\mbr^n;\;\;\phi^L(s,\gamma)= (t\to \phi_s(\gamma(t)))\end{equation}

\begin{proposition} $\phi^L$ as defined above is an isotopy  on $L\mbr^n$. For $X\in T_\gamma L\mbr^n\approx C^\infty(S^1,\mbr^n)$, derivative of $\phi^L_s$ at $\gamma$ is given by $d\phi^L_s(\gamma)(X)= \left(t\to d\phi_s(\gamma(t))(X(t))\right).$
\end{proposition}
\begin{proof}
Take any smooth curve $c:\mathbb R\to \mathbb R\times  L\mbr^n$, $c(u)= (c_1(u), c_2(u))$. Then
\begin{eqnarray*}
\tilde{c}(u)&:& \mbr\to L\mbr^n\text{ defined by }\\
\tilde{c}(u)&=& \phi^L\circ c= \phi^L(c_1(u), c_2(u))\\
\tilde{c}(u)&=& (t\to \phi_{c_1(u)}(c_2(u)(t))).
\end{eqnarray*}
By the definition of smooth map in $L\mbr^n$, $\tilde{c}$ is smooth if and only if $\tilde{c}^\vee:\mbr\times S^1\to\mbr^n$ defined by, $\tilde{c}^\vee(u,t)\to \phi(c_1(u),c_2(u)(t))$ is smooth. As $\phi$  and $c_1$, $c_2$ smooth map, we have $\tilde{c}$ is a smooth map.

This shows that for  every $c: \mbr\to R\times L\mbr^n$, we have $\phi^L\circ c$ is smooth.  Therefore by the definition of smooth map, which we fixed, we see that $\phi^L$ is a smooth map.

For $s\in \mbr$, $\phi_s$ is a diffeomorphism. Let $\psi_s$ is the inverse of $\phi_s$. Let $\psi^L_s$ is defined in the similar way.  The inverse of $\phi^L_s$ will  $\phi^L_s$. Therefore we $\phi_s^L$ is a diffeomorphism.

Also we have $\phi^L_0= Id_{L\mbr^n}$.  This proves that $\phi^L$ is an isotopy of $L\mbr^n$.

Now we will calculate the derivative of the map $\phi^L_s$ for fixed $s\in \mbr$;
$$\phi^L_s:L\mbr^n\to L\mbr^n\;\; ;\phi^L_s(\gamma)(t)=\phi_s(\gamma(t))$$
Let $\gamma\in L\mbr^n$, then $d\phi^L_s(\gamma):T_\gamma L\mbr^n\to T_{\phi^L_s(\gamma)}L\mbr^n$.  Take $X\in T_\gamma L\mbr^n\approx C^\infty(S^1,\mbr^n)$,  $d\phi^L_s(\gamma)(X)$ is a vector field along $\phi^L_s(\gamma)$. Derivative is determined by directional derivative.

Identify $T_\gamma L\mbr^n$ with $C^\infty(S^1,\mbr^n)$. $d\phi^L_s(\gamma)(X)$ is an element of $T_{\phi^L_s(\gamma)}L\mbr^n\approx C^\infty(S^1,\mbr^n)$ which is the limit of net $\left\{\frac{\phi^L_s(\gamma+uX)-\phi^L(\gamma)}{u}\right\}$. Evaluation at time $t$ is a continuous linear map $L\mbr^n\to \mbr$ which takes this net to
\begin{eqnarray*}
&&\left\{\frac{\phi^L_s(\gamma+uX)(t)-\phi^L_s(\gamma)(t)}{u}\right\}\\
=&&\left\{\frac{\phi_s(\gamma(t)+uX(t))-\phi_s(\gamma(t))}{u}\right\}
\end{eqnarray*}
This is a differential quotient which tends to $d\phi_s(\gamma(t))(X(t))$.  Therefore we have
\begin{equation}\label{derivativeofphiL}d\phi_s^L(\gamma)(X)= (t\to d\phi_s(\gamma(t))(X(t)))
\end{equation}
\end{proof}

We will be needed the following function in next sections.  Let us define a map on loop space corresponding to a map on base manifold.
If $f:\mbr^n\to \mbr$ is $C^\infty$ function then define
\begin{eqnarray}\label{functionf}&&\label{f}\tilde{f}: L\mbr^n\to \mbr\text{ by}\\
&&\tilde{f}(\gamma)=\int_0^1f(\gamma(t)) dt\nonumber
\end{eqnarray}
Then using the same argument as above, we have for $\gamma\in L\mbr^n$, $X\in T_\gamma(L\mbr^n)\approx C^\infty(S^1, \mbr^n)$,
\begin{equation}\label{derivativeoff}d\tilde{f}_\gamma(X)= \int_0^1 df_{\gamma(t)}(X(t))dt\end{equation}

\section{Symplectic structure on Loop space}
Let $\omega$ be a symplectic form on a finite dimensional manifold $M$.  Define for each $\gamma\in LM$, and $X,Y\in T_\gamma LM \approx L\mbr^n$,
\begin{equation}\label{omegaalpha}\Omega^\omega_\gamma(X,Y)=\int \omega_{\gamma(t)}(X(t),Y(t))dt.\end{equation}

Observe that the $\Omega^\omega$, a 2-form on $LM$ is same as  Chen's ``Iterated Integral" \cite{chen}. Therefore it is a smooth section and non degenerate (for details, we refer \cite{chen,indranil}).  For calculating the cohomology of loop space, Chen in \cite{chen} defines a graded differential algebra:  with spaces consisting differential form arising from  iterated integral and a special type of derivative which he called exterior derivative $d_{chen}$ on iterated integral differential form. For full detail, we refer to \cite{chen}.  We have
\begin{equation}(d_{chen}\Omega^\omega)_\gamma(X,Y,Z)=\int (d\omega)_{\gamma(t)}(X(t),Y(t),Z(t))dt =0.\end{equation}

For the form $\Omega^\omega$, this exterior derivative definition matches with the definition of general exterior derivative of smooth form in Fr\'{e}chet space.  See definition 33.12 of \cite{KM},  which we have to use for computing $d(\Omega^\omega)$.  By a obvious calculation we see that
$d(\Omega^\omega)= d_{chen}\Omega^\omega$.

As $\omega$ is closed, $(d\omega)=0$, we have $d(\Omega^\omega)=0$.  Therefore $\Omega^\omega$ is closed 2-form in the general sense [Kriegl and Michor sense].

Also $\Omega^\omega$ is non-degenerate is clear from \cite{chen, indranil}.  Hence $\Omega^\omega$ is weak symplectic structure on $LM$.

We make a remark that if $J$ is an almost complex structure on $M$ compatible with the  symplectic structure $\omega$ on $M$.  Following \cite{pradip},\cite{indranil}, if we define an endomorphism $\mathfrak{J}$ of $T_\gamma LM$ by
$\mathfrak{J}s_1= s_2$, if $J(s_1(t))= s_2(t)$ for all $t$.  Then this is an almost complex structure on $LM$ compatible with $\Omega^\omega$.

\section{Darboux chart on $L\mbr^n$}
In 1969, for symplectic Banach manifolds,  Weinstein \cite{weinstein} proved the existence of Darboux chart for strong symplectic Banach manifolds. In 1972, Marsden \cite{marsden} showed that Darboux chart may not exists for weak symplectic Banach manifold.  In 1999 Bambusi \cite{bambusi} has given the necessary and sufficient condition for existence of  Darboux charts for weak symplectic Banach manifold (in the case when model space are reflexive).  For a weak symplectic Fr\'{e}chet manifold, the general existence theorem is not known.  In \cite{pradip3},  we  prove a general existence theorem for a special type of Fr\'{e}chet manifold called PLB (projective limit of Banach manifold) with a special type of weak symplectic structure.
In this article when Fr\'{e}chet manifold is $L\mbr^n$ and symplectic structure is $\Omega^\omega$, then situation become almost trivial.

Denote $\Omega^s:= \Omega^{\omega_s}$ then we have following theorem,
\begin{theorem}\label{thm:isotopyalpha} If $\omega_s$ are symplectic form on $\mathbb R^n$ and suppose there is an isotopy $\phi$ on $\mbr^n$ such that $\phi^*_s(\omega_s)=\alpha_0$, then there is an isotopy $\phi^L$ on Loop space such that $(\phi^L)^*(\Omega^s)= \Omega^0$
\end{theorem}
\begin{proof}
We have given isotopy $\phi:\mbr\times \mbr^n\to \mbr^n$ such that
$(\phi^s)^*(\omega^s)= \omega^0$,  Then as in section 2, define
\begin{eqnarray*}
&&\phi^L:\mbr\times L\mbr^n\to L\mbr^n\\
\text{by }&&\phi^L_s(\gamma)= (t\to \phi^s(\gamma(t)));
\end{eqnarray*}
we have seen already that by equation \ref{derivativeofphiL}, $d\phi_s^L(\gamma)(X)= (t\to d\phi_s(\gamma(t))(X(t))).$
Therefore we have
\begin{eqnarray*}(\phi^L_s)^*(\Omega^s)_\gamma(X,Y)
&=&\Omega^s_{\phi^L_s(\gamma)}(d\phi_s^L(\gamma)(X),d\phi_s^L(\gamma)(Y))\\
&=&\int\omega^s_{\phi_s(\gamma(t))}(d\phi_s(\gamma(t))(X(t)),d\phi_s(\gamma(t))(Y(t)))dt\\
&=&\int \omega^0_{\gamma(t)}(X(t),Y(t))dt\\
&=&\Omega^0(X,Y)
\end{eqnarray*}
\end{proof}

Theorem \ref{thm:isotopyalpha} concludes that for any $\Omega^\omega$, $L\mbr^n$ admits Darboux chart.

\section{Cohomology of loop space $L\mbr^n$}
By the symplectic cohomology, we mean the definition given by Kriegl and Michor in \cite{KM}.  We will follow  notations and definitions of section 48 of \cite{KM}.  For sake of completeness, below we will define the required terms.

Let $(\mM, \sigma)$ be a weak symplectic Fr\'{e}chet manifold.  Let $T_x^\sigma \mM$ denotes the real linear subspace $T_x^\sigma \mM= \sigma_x^\vee(T_x \mM)\subset T_x^*(\mM)$. These vector space fit to form a sub bundle of $T^*\mM$ (\cite{KM},\S48.4) and $\sigma^\vee:T\mM\to T^\sigma \mM$ is vector space isomorphism.   Define $C^\infty_\sigma(\mM,\mbr)\subset C^\infty(\mM,\mbr)$ be the linear subspace consisting of all smooth functions $f:\mM\to \mbr$ such that $df:\mM\to T^*\mM$ factors to a smooth mapping $\mM\to T^\sigma \mM$.

In other words,  $f\in C^\infty_\sigma(\mM,\mbr)$ if there exists a smooth $\sigma$-gradient $grad^\sigma f\in \mathfrak{X}(\mM)$ such that for given $p\in \mM$ and $Y\in T_p\mM$ we have
$df_p(Y)= \sigma_p(grad^\sigma f|_p, Y)$.  Detail description of these spaces and analysis of weak symplectic manifold is given in (\cite{KM},\S48)

Let $C^\infty(L^k_{alt}(T\mM,\mbr)^\sigma)$ be the space of smooth sections of vector bundle with fiber $L^k_{alt}(T_x\mM,\mbr)^{\sigma_x}$ consisting of all bounded skew symmetric  forms $\omega$ with $\omega(., X_2,..., X_k)\in T_x^\sigma \mM$.
Let $$\Omega_\sigma^k(\mM):= \{\omega\in C^\infty(L_{alt}^k(T\mM,\mbr)^\sigma): d\omega\in C^\infty(L_{alt}^{k+1}(T\mM,\mbr)^\sigma)\}.$$
Since $d^2=0$ and wedge product of $\sigma$ dual forms is again a $\sigma$- dual form [see page 527 of \cite{KM}]. We have a graded differential subalgebra $(\Omega_\sigma(\mM),d)$, whose cohomology is called the symplectic cohomology and  will be denoted by $H^k_\sigma(\mM)$.  We again mention that this definition of the symplectic cohomology is not same as the symplectic cohomology defined by Floer.

Following [\S34,\cite{KM}], for a smooth convenient (in particular Fr\'{e}chet) manifold $\mM$ consider a graded algebra $$\Omega(\mM)=\oplus_{k=0}^\infty \Omega^k(\mM)$$
$d(\phi\wedge \psi)= d\phi\wedge \psi+ (-1)^{deg(\phi)}\phi\wedge d\psi$.  Define
\begin{equation}\label{defn:derhamcohomology}
H^k_{DR}(\mM)=\frac{\{\omega\in \Omega^k(\mM): d\omega=0\}}{\{d\phi: \phi\in \Omega^{k-1}(\mM)\}}
\end{equation}
$H^k_{DR}(\mM)$ is called $k$-th  de-Rham cohomology of $\mM$.

Our aim of this section is to show for loop space $\mM= L\mbr^n$ with symplectic structure $\sigma= \Omega^\omega$, $$H^k_\sigma(LM)\simeq\ H^k_{DR}(LM)$$

For this first we have the following proposition:
\begin{proposition}\label{existenceofpartitionofunity}
$(LM,\sigma)$ admits smooth partition of unity in $C_\sigma^\infty(LM, \mbr)$.
\end{proposition}
\begin{proof}
As $M$ is embedded in some $R^n$, $LM$ is embedded in $L\mbr^n$ (\cite{stacey2}). As  $L\mbr^n$ is metric space, therefore  $LM$ admits smooth partition of unity in $C^\infty(LM,\mbr)$ but we are needed to show that $LM$ admits smooth partition of unity in $C_\sigma^\infty(LM,\mbr)$.

Let $\{\mathcal{U}_\alpha\}$ be a covering of $LM$. For each $\mathcal{U}_\alpha$, define
$$U_\alpha:= \{p=\gamma(t) \text{ for some } \gamma\in \mathcal{U}_\alpha\text{ and } t\in S^1\}$$

As for the loop space $LM$, $ev_t: LM\to M$ is open map for any $t$.  Therefore $U_\alpha$ is open subset of $M$ and this collection covers $M$. Let $\{f_\alpha\}$ be  the partition of unity subordinate to the covering $\{U_\alpha\}$,  then for each $\alpha$, as equation \ref{functionf}, define
$$\hat{f}_\alpha(\gamma)= \int_0^1f(\gamma(t)) dt$$
We see that the support of $\hat{f}_\alpha\subset \mathcal{U}_\alpha$ as support of $f_\alpha \subset U_\alpha$.
Also for each $f_\alpha$, there is a vector field $X_\alpha$ on $M$ such that $df_\alpha= \omega(X_\alpha, .)$.  Let $\hat{X}_\alpha$ be a vector field on $LM$ defined by $\hat{X}_\alpha(\gamma)(t):= X_\alpha(\gamma(t))$.   Call this vector field on $LM$ a associated vector field with $X_\alpha$.

From equation \ref{derivativeoff},  we  see that for $X\in T_\gamma LM$, $(d\hat{f}_\alpha)_\gamma(X)= \int_0^1 df_{\gamma(t)}(X(t))dt$.
This gives
$$(d\hat{f}_\alpha)_\gamma(X)= \int_0^1\omega_{\gamma(t)}(X_\alpha(\gamma(t)),.)$$
That is $$(d\hat{f}_\alpha)_\gamma(X)= \Omega^\omega(\hat{X_\alpha},.)$$
This proves $$\hat{f_\alpha}\subset C_\sigma^\infty(LM,\mbr)$$
Also $\sum_\alpha \hat{f}_\alpha= 1$.  Hence collection $\{\hat{f}_\alpha\}$ is the required partition of unity.
\end{proof}

We recall the theorem from \cite{KM}:
\begin{theorem}[\S48.9,\cite{KM}]\label{thmKM}
If $(M,\sigma)$ is a smooth weakly symplectic manifold which admits smooth partitions of unity in $C_\sigma^\infty(M,\mbr)$, and which admits `Darboux chart', then the symplectic cohomology equals the De Rham cohomology: $H^k_\sigma(M)= H^k_{DR}(M)$.
\end{theorem}

Combining above theorem \ref{thmKM} with proposition \ref{existenceofpartitionofunity} and theorem \ref{thm:isotopyalpha} we conclude that for $(LM,\Omega^\omega)$, we have
$$H^k_{DR}(L\mbr^n)\simeq H^k_{\Omega^\omega}(L\mbr^n).$$

\begin{remark}[Further direction]
Proposition \ref{existenceofpartitionofunity} is not only true for $(L\mbr^n, \Omega^\omega)$ but it is true for any $(LM, \Omega^\omega)$.   But in calculating the derivative of $\phi^L_s$ in section 2 and proving that $(L\mbr^n, \Omega^\omega)$ admits Darboux chart, we used that $\mbr^n$ admits a global Darboux chart for any symplectic form $\omega$.  For general loop space $LM$ in place of $L\mbr^n$, the proof given in section 2, 4 is not working.

Still we believe that $(LM,\Omega^\omega)$ will admit Darboux chart for most of symplectic manifolds $(M,\omega)$. In this direction in \cite{pradip3}, we  studied a general class of Fr\'{e}chet manifolds, called PLB manifold and necessary conditions for existence of Darboux chart on weak symplectic PLB (projective limit of Banach) manifolds.  In the terminology of article \cite{pradip3},  $LM$ is a PLB manifold and $\Omega^\omega$ is a compatible symplectic structure with projective system.  We hope that for some symplectic manifold $(M,\omega)$ and loop space over that $M$,  $(LM,\Omega^\omega)$ may admit Darboux chart.   Suppose for some $(M,\omega)$, $(LM,\Omega^\omega)$ admits Darboux chart,  for those $M$, we will have $H^k_{DR}(L\mbr^n)\simeq H^k_{\Omega^\omega}(L\mbr^n)$ for any $k\geq 0$.
\end{remark}

\end{document}